\documentclass{article}

\usepackage{amsmath,amssymb}
\usepackage{url}
\usepackage[utf8]{inputenc}
\usepackage{array}
\usepackage{caption}
\usepackage{enumitem}

\newcommand     {\R}            {{\mathbb R}}
\newcommand     {\C}            {{\mathbb C}}

\newcommand     {\SL}[1][\Z]    {\sym{SL}(2,{#1})}
\newcommand     {\Z}            {{\mathbb Z}}

\newcommand     {\sym}[1]       {\operatorname{#1}}

\newcommand     {\abs}[1]       {{\left\lvert{#1}\right\rvert}}
\newcommand     {\kro}[2]       {{\left(\frac{#1}{#2}\right)}}
\newcommand     {\smat}[1]      {{\begin{smallmatrix}#1\end{smallmatrix}}}

\newcommand{\GL}[2]{GL(\ensuremath{#1},\ensuremath{#2})}
\newcommand{\Sp}[2]{Sp(\ensuremath{#1},\ensuremath{#2})}
\newcommand{\GSp}[2]{GSp^+(\ensuremath{#1},\ensuremath{#2})}

\renewcommand{\H}{\mathcal{H}}

\newcommand{\para}[1]{\Gamma^{\text{para}}[#1]}

\DeclareMathOperator{\disc}{\text{disc}}

\usepackage{amsthm}
\theoremstyle{plain}
\newtheorem{lemma}{Lemma}
\newtheorem{proposition}[lemma]{Proposition}
\newtheorem{theorem}[lemma]{Theorem}

\newtheorem*{conja}{Conjecture A}
\newtheorem*{conjb}{Conjecture B}
\newtheorem*{conjc}{Conjecture C}

\numberwithin{lemma}{section}

\theoremstyle{definition}

\newtheorem{remark}{Remark}

\title{%
  Formulas for central critical values of twisted
  L-functions attached to paramodular forms
}

\author{%
Nathan C. Ryan
\\ Department of Mathematics, Bucknell University
\\ nathan.ryan@bucknell.edu
\and
Gonzalo Tornar\'ia
\\ Centro de Matem\'atica, Universidad de la Rep\'ublica
\\ tornaria@cmat.edu.uy
}

\date{%
  \today}

\begin{document}

\maketitle

\begin{abstract}
In the 1980s B\"ocherer formulated a conjecture relating the central
values of the imaginary quadratic twists of the spin L-function attached to a
Siegel modular form $F$ to the Fourier coefficients of $F$.  This conjecture
has been proved when $F$ is a lift.  More recently, we formulated in~\cite{RyanTornaria} an
analogous conjecture for paramodular forms $F$ of prime level, even weight and in
the plus-space. \par In this paper, we
examine generalizations of this conjecture.  In particular, our
new formulations relax the conditions on $F$ and allow for twists by real
characters.  Moreover, these formulation are more explicit than the
earlier version.  We prove the conjecture in the case of lifts and
provide numerical evidence in the case of nonlifts.
\end{abstract}

\section{Introduction}\label{sec:intro}

Many problems in number theory are related to central values of
L-functions associated to modular forms,
and central values of twisted L-functions are tools
used to make progress on these problems.  In this paper we focus our
attention on paramodular forms of level $N$ and the spin
L-function associated to them.

In the 1980s a conjecture was formulated by B\"ocherer \cite{Bocherer}
that relates the coefficients of a Siegel modular form $F$ of degree 2 and
the central value of the spin L-function associated to $F$.  One
fixes a discriminant $D$ and, roughly speaking, adds up all the
coefficients of $F$ indexed by quadratic forms of discriminant $D$.  One also
computes the central value of the spin L-function twisted by the
quadratic character $\chi_D$.  The conjecture asserts that the central
value, up to a constant that depends only on $F$ (and not on $D$) and
a suitable normalization, is the square of the sum of coefficients.
In B\"ocherer's original paper \cite{Bocherer} it was proved for $F$ that are
Saito-Kurokawa lifts and later B\"ocherer and  Schulze-Pillot
\cite{BochererSchulzePillot} proved the conjecture in the case when $F$ is a Yoshida
lift.  Kohnen and Kuss \cite{KohnenKuss} gave numerical evidence in the
case when $F$ is of level one, degree 2 and is not a Saito-Kurokawa
lift (these computations have recently been extended by Raum
\cite{Raum}).  A much more general approach to the conjecture has been
pursued by Furusawa, Martin and Shalika
\cite{Furusawa,FurusawaMartin,FurusawaShalika2,FurusawaShalika3,FurusawaShalika1}.

In \cite{RyanTornaria} we investigated an analogous
conjecture in the setting of paramodular forms and our goal here is to
state some generalizations of the conjecture and to point out some
subtleties in the statement of the original conjecture.  Paramodular forms have been
gaining a great deal of attention because, for example, the most
explicit analogue of Taniyama-Shimura for abelian surfaces has been
formulated by Brumer and Kramer~\cite{BrumerKramer} and been verified
computationally by Poor and Yuen~\cite{PoorYuen}.  

We summarize the main results in \cite{RyanTornaria}.  Fix a
paramodular eigenform $F$ of level $N$ with Fourier coefficient $a(T;F)$
for each positive semidefinite quadratic form $T$.  One defines 
\[
A(D) = A_F(D) :=\frac{1}{2}\sum_{\{T>0\;:\;\disc T=D\}/\Gamma_0(N)}\frac{a(T;F)}{\varepsilon(T)}
\]
where $\varepsilon(T):=\# \{U\in\Gamma_0(N):T[U]=T\}$.  We
provided evidence for a conjecture, which can be
considered a generalization of Waldspurger's theorem
\cite{Waldspurger}.  We state a different version of that conjecture now:

\begin{conja}[Paramodular B\"{o}cherer's Conjecture]
Let $p$ be prime and let $F\in S^k(\para{p})^+$ be a paramodular Hecke eigenform of even weight $k$.
Then, for fundamental discriminants $D<0$ we have
\begin{equation}\label{eq:bocherer}
  A_F(D)^2 =  \alpha_D\,  C_F \, L(F,1/2,\chi_D)\,\abs{D}^{k-1}
\end{equation}
where $C_F$ is a nonnegative constant that depends only on $F$,
and where $\alpha_D=1+\kro{D}{p}$.
Moreover, when $F$ is a Gritsenko lift we have $C_F>0$,
and when $F$ is not a lift, we have $C_F=0$ if and only if $L(F,1/2)=0$.
\end{conja}

This conjecture corrects a defect of the corresponding
conjecture in \cite{RyanTornaria}, the defect being that it might be
wrong in the case of nonlifts since the conjecture in
\cite{RyanTornaria} requires the constant be positive.
It is a
theorem in \cite{RyanTornaria} that $C_F>0$ when $F$ is a Gritsenko
lift and we know that $C_F>0$ in all the examples of nonlifts that we
computed.
A minor difference is that the formula here uses the factor
$\alpha_D$ instead of the factor $\star$ in the previous version of
the conjecture.  Though $\alpha_D$ here can vanish (while $\star$
could not), this does not affect the conjecture since
$\kro{D}{p}=-1$ implies $A(D)$ is an empty sum and $L(F,1/2,\chi_D)=0$ because
in this case the functional equation has sign $-1$.
However, this factor $\alpha_D$ will be essential in the case of composite levels.

\begin{remark}We make the simple observation that if Conjecture~A is
  true and if $F$ is a nonlift for which $L(F,1/2)=0$, then the average
  $A(D)$ of its Fourier coefficients would be zero for all $D$.  This is a
  step in characterizing the kinds of forms that might violate the
  conjecture as stated in \cite{RyanTornaria}.
\end{remark}

\begin{remark}Our Conjecture~A has a noteworthy difference from the
  original version of the conjecture first stated by B\"ocherer in
  \cite{Bocherer}.  Namely, in his conjecture the constant $C_F$ is
  required to be positive while we only require that ours be
  nonnegative, though we do characterize exactly when the constant is
  zero.  This was subsequently addressed in a limited way in
  \cite{BochererSchulzePillot}, where the conjecture for Siegel
  modular forms of level $N$ was considered.

\end{remark}

\begin{remark} In order to verify Conjecture~A in a case where $F$ is
  a nonlift and $L(F,1/2)=0$, one would have to compute the Fourier
coefficients of  a paramodular
form whose L-function vanishes to even order greater that zero.  Here we note that if
the Paramodular Conjecture holds, there should be such a paramodular
of, for example, levels 3319, 3391, 3571, 4021, 4673, 5113, 5209,
5449, 5501, 5599  since there is an hyperelliptic curve for each of these conductors
whose Hasse-Weil L-function vanishes to even order at least
two~\cite{Stoll}.  This last assertion
about the order of vanishing was verified
by directly computing the central
value of these L-functions using $\text{lcalc}$ \cite{lcalc}.
\end{remark}

\subsection{Two surprises}

After carrying out the computations used to verify the conjecture in \cite{RyanTornaria}, we made two
observations that we describe now.  We asked about what happens
if you do not restrict the computations to forms in the plus space.
To do this, we first noticed that the two sides of \eqref{eq:bocherer}
do indeed make sense.  We also carried out a simple computation
\cite[Section 4]{RyanTornaria} that
shows the averages $A(D)$ add to zero when a form is in the minus
space.  Undaunted, we carried out the computations and tabulated the
following data for $F_{587}^{-}$ (see Table~\ref{tbl:egs} for a list
of all the examples considered in this paper).
\begin{center}
\begin{tabular}{c|ccc ccc ccc ccc}
    $D$
    &  -4 &  -7 & -31
    & -40 & -43 & -47
    \\\hline
    $L_D/L_{-3}$
    &   1.0 &   1.0 &   4.0
    &   9.0 & 144.0 &   1.0
\end{tabular}
\end{center}
Here $L_D := L(F_{587}^{-},1/2,\chi_D)\cdot\abs{D}$ and the table
shows fundamental discriminants for which $\kro{D}{587} = -1$.
The obvious thing to notice is that the numbers in the table appear to be squares and so the natural question to
ask is: squares of what?

This first experiment was a natural extension of our previous work in
\cite{RyanTornaria} as we had a paramodular form to compute the
righthand side of \eqref{eq:bocherer} and both sides of the equation
make sense.  Emboldened by the results of the first experiment, we
decided to change another hypothesis in the conjecture: we decided
to look at the case when $D>0$.  This is somewhat unnatural as the
sum $A(D)$ is an empty sum in this case.
Nevertheless we get the following data for $F_{277}$.
\begin{center}
\begin{tabular}{c|ccc ccc ccc ccc}
    $D$
    & 12 & 13 & 21
    & 28 & 29 & 40
    \\\hline
    $L_D/L_{1}$
    &  225.0 &  225.0 &  225.0
    &  225.0 & 2025.0 &  900.0
\end{tabular}
\end{center}
Here $L_D := L(F_{277},1/2,\chi_D)\cdot\abs{D}$ and $\kro{D}{277}=+1$.  Again, these seem
to be squares, but squares of what? (Also, the observant reader may
have noticed that all these squares are divisible by $15^2$.  See
Section~\ref{sec:torsion} for more about this.)

In Section~\ref{sec:conj} we will show how to define, for an auxiliary
discriminant $\ell$, a twisted average $B_\ell(D)$.
When $\ell$ is properly chosen, the squares of $B_\ell(D)$ are exactly
the squares we see in the previous two tables.

Given a discriminant $\Delta$, we put
\[
   \alpha_\Delta := \prod_{p\mid N} \left(1+\kro{\Delta_0}{p}\right) \;,
\]
where $\Delta_0$ is the fundamental discriminant associated to $\Delta$.

\begin{conjb}
Let $N$ be squarefree.  Suppose $F\in S^k(\para{N})$ is a Hecke
eigenform and not a Gritsenko lift.
Let $\ell$ and $D$ be fundamental discriminants
such that $\ell D<0$.
Then
\begin{equation*}
B_{\ell,F}(D)^2=
\alpha_{\ell D} \, C_{\ell,F} \, L(F,1/2,\chi_D) \, \abs{D}^{k-1}
\;,
\end{equation*}
where $C_{\ell,F}$ is a constant independent of $D$.
Moreover, $C_{\ell,F}=0$ if and only if $L(F,1/2,\ell)=0$.
\end{conjb}

The notation in this conjecture is further explained  in Section~\ref{sec:conj}
but the analogy with Conjecture~A will be made clear now.  First, note
that $B_\ell(D)$ is a twisted average of the Fourier coefficients of
$F$ indexed by quadratic forms of discriminant $D$.
Essentially, if $k$ is even, $N$ is prime, and $\ell=1$ then
$B_\ell(D)=\abs{A(D)}$ so we recover Conjecture~A.

In a later section, we are interested in verifying Conjecture~B in
the case of nonlifts.  To do this, we attempt to understand the constant
$C_{\ell,F}$ a little better.  In Conjecture~B, one can think of the
discriminant $\ell$ as being fixed.  In this next conjecture, we think
of it as a parameter.

\begin{conjc}
Let $N$ be squarefree.  Suppose $F\in S^k(\para{N})$ is a Hecke
eigenform and not a Gritsenko lift. 
Let $\ell$ and $D$ be fundamental discriminants such that $\ell D<0$.
Then
\begin{equation*}
B_{\ell,F}(D)^2=
\alpha_{\ell D}
\, k_F
\, L(F,1/2,\chi_\ell) \, L(F,1/2,\chi_D)
\, \abs{D\ell}^{k-1}
\end{equation*}
for some positive constant $k_F$ independent of $\ell$ and $D$.
\end{conjc}

This gives us a very explicit statement of a conjecture for forms that
are not Gritsenko lifts.  It is this formula that we verify in
Section~\ref{sec:nonlifts}.  We observe that
Conjecture~C implies Conjecture~B by letting
$C_{\ell,F}=k_F \, L(F,1/2,\chi_\ell)\,\abs{\ell}^{k-1}$.

We note that when $F$ is a Gritsenko lift the formula of Conjecture~B
is valid in the case $\ell=1$ with $C_{\ell,F}>0$, as shown in
Theorem~\ref{thm:grit} below;
the formula of
Conjecture~C is valid provided $\ell\neq 1$ and $D\neq 1$,
but uninteresting
with both sides being zero for trivial reasons
(see Proposition~\ref{prop:avg_lift} and  Proposition~\ref{prop:centralvalue_lift}).
\\

\subsection{Notation}  
The main objects of study in this paper are paramodular forms of 
level $N$ and their L-functions.

Suppose $R$ is a commutative ring with identity. The symplectic group
is $\Sp{4}{R}:=\{x\in \GL{4}{R}:  x' J_2 x = J_2\}$,
where the transpose of matrix $x$ is denoted $x'$ and for the $n
\times n$ identity matrix $I_n$ we set
$J_n = \left(\begin{smallmatrix}
0 & I_n\\-I_n&0
\end{smallmatrix}\right)$.
When $R\subset \R$, the group of symplectic similitudes
is $\GSp{4}{R} := \{x\in\GL{4}{R}: \exists \mu\in\R_{>0}: x' J_2 x =
\mu J_2\}$.

The paramodular group of level $N$ is
\begin{equation*}
\para{N} := \Sp{4}{\mathbb{Q}}\cap 
\begin{pmatrix} 
* &* & */N &*\\
N* & * &*& *\\
N*& N*& * & N*\\
N* & * & * & *
\end{pmatrix},
\text{ where $*\in\Z$.}
\end{equation*}

Paramodular forms of degree 2, level $N$ and weight $k$ are modular
forms with respect to the group $\para{N}$.  We denote the space of
such modular forms by $M^k(\para{N})$ and the space of cuspforms by
$S^k(\para{N})$.  The space $S^k(\para{N})$ can be split into a plus
space and a minus space according to the action of the Atkin-Lehner
operator $\mu_N$: in particular, $S^k(\para{N})^\pm =\{f\in S^k(\para{N}): f\mid \mu_N = \pm f\}$.  

Every $F\in M^k(\para{N})$ has a Fourier expansion of the form
\[
F(Z) = \sum_{T=[Na,b,c]\in \mathcal{Q}_N} a(T;F)\,q^{Na}\,\zeta^b\,{q'}^{c}
\]
where $q := e^{2\pi i z}$, $q':=e^{2\pi i z'}$ ($z,z'\in \H_1$), $\zeta := e^{2 \pi i \tau}$ ($\tau\in\C$)
and
\[
\mathcal{Q}_N := \left\{[Na,b,c]
\geq 0
\;:\;
a,b,c\in\Z
\right\};
\]
here we use Gauss's notation for binary quadratic forms.

We will want to decompose $\mathcal{Q}_N$ by discriminant $D< 0$ so we also
define
\[
\mathcal{Q}_{N,D} = \left\{ T\in \mathcal{Q}_N: \disc T = D\right\}
.
\]
This is useful, for example, so that we can write
\[
A_F(D):=\frac{1}{2}\sum_{T\in \mathcal{Q}_{N,D}/\Gamma_0(N)}\frac{a(T;F)}{\varepsilon(T)}
\;.
\]

For $F\in S^k(\para{N})$, we have $a(T[U];F)=a(T;F)$ for
every $U\in\Gamma_0(N)$,
where $\Gamma_0(N)$ is the
congruence subgroup of $\SL$ with lower lefthand entry congruent to 0
mod $N$, and $a(T[\smat{1&0\\0&-1}];F)=(-1)^k\,a(T;F)$.
Moreover, cusp forms are supported on the positive definite
matrices in $\mathcal{Q}_N$.






Suppose we are given a paramodular form $F\in S^k(\para{N})$ so that
for all Hecke operators $T(n)$, $F|T(n) = \lambda_{F,n}F=\lambda_n F$.  Then we can
define the spin L-series by the Euler product
\begin{equation*}\label{eq:spin}
L(F,s) := \prod_{\text{$q$ prime}} L_q\bigl(q^{-s-k+3/2})^{-1},
\end{equation*}
where the local Euler factors are given by
\[
L_q(X) := 1 - \lambda_q X + (\lambda_q^2-\lambda_{q^2}-q^{2k-4}) X^2
            - \lambda_q q^{2k-3} X^3 + q^{4k-6} X^4
\]
for $q\nmid N$, and has a similar formula but of lower degree for
$q\mid N$.

The Paramodular Conjecture \cite{BrumerKramer,PoorYuen} asserts that
the L-function of a para\-modular form is the same as the L-function
of an associated abelian surface.
In all the examples we consider, the paramodular forms have
corresponding abelian surfaces isogenous to Jacobians of hyperelliptic
curves that can be found in tables of Stoll \cite{Stoll};
thus we compute hyperelliptic curve L-functions when we carry out our computations.
A table in \cite{Dokchitser} summarizes the data that we use to write
down the functional equation of the L-function of an hyperelliptic
curve: 
\begin{equation*}
L^*(F,s) = \left(\frac{\sqrt{N}}{4\pi^2}\right)^s\Gamma(s+1/2)\Gamma(s+1/2)L(F,s).
\end{equation*}
so that conjecturally
\[
L^*(F,s) = \epsilon\, L^*(F,1-s),
\]
when $F\in S^2(\para{N})^\epsilon$.

Let $D$ be a fundamental discriminant,
and denote by $\chi_D$ the
unique qua\-dratic character of conductor $D$.  For the spin
L-series $L(F,s) = \sum_{n\geq 1} a(n)\,n^{-s}$ of a paramodular form
$F$, we define the quadratic twist
\[
L(F,s,\chi_D) := \sum_{n \geq 1} \chi_D(n)\,a(n)\,n^{-s},
\]
which is conjectured to have an analytic continuation and satisfy a
functional equation. Suppose $N$ is squarefree, let $N'=ND^4/\gcd(N,D)$,
and define
\[
L^*(F,s,\chi_D) :=\left(\frac{\sqrt{N'}}{4\pi^2}\right)^s\Gamma(s+1/2)\,\Gamma(s+1/2)\,L(F,s,\chi_D)
\]
so
that assuming standard conjectures, 
\[
L^*(F,s,\chi_D)=\epsilon'\,L^*(F,s,\chi_D).
\]
The global root number $\epsilon'$ of the functional equation for $L(F,s,\chi_D)$ is
given in terms of the local root numbers $\epsilon_p$ of $L(F,s)$ by the following lemma \cite{SchmidtNotes}.


\begin{lemma}
Let $F\in S^k(\para{N})$ be a Hecke eigenform, with $N$ squarefree.
Denote the local root numbers of, respectively,
$L(F,s,\chi_D)$ and
$L(F,s)$, as follows: $\epsilon'=\prod_{p\leq \infty}
\epsilon_p'$, and $\epsilon=\prod_{p\leq \infty} \epsilon_p$.  Then
\begin{enumerate}
\item At the infinite place, $\epsilon_\infty' =\epsilon_\infty = (-1)^k$.
\item Assume $p\nmid N$, then $\epsilon_p'=\epsilon_p=+1$.
\item Assume $p\mid N$ and $p\mid D$, then $\epsilon_p'=+1$.
\item Assume $p\mid N$ and $p\nmid D$, then $\epsilon_p'=\chi_D(p)
  \epsilon_p$.
\end{enumerate}
In particular, if $N_0=N/\gcd(N,D)$,
\[
\epsilon' = \epsilon \cdot \chi_D(N_0) \cdot \prod_{p\mid\gcd(D,N)} \epsilon_p
\]
\end{lemma}

\section{Generalizations of the Paramodular B\"ocherer's Conjecture}\label{sec:conj}

In this section, we motivate Conjectures~B and C.  We do it by
describing what happens for particular $F$ that are not Gritsenko
lifts.  We place particular emphasis on the transition from the
hypotheses in Conjecture~A (namely, $F$ in the plus-space, of prime
level and even weight) to Conjectures~B and C which have no such
hypotheses.



\subsection{A simple case}

For $F=F_{249}$, the unique Hecke eigenform of weight 2 and level
$249=3\cdot 83$, we have $A(D)=0$ for all $D$ (see Lemma~\ref{lem:AD},
below) since its eigenvalues under the Atkin-Lehner opertors $\mu_3$
and $\mu_{83}$ are $\epsilon_3=\epsilon_{83}=-1$.
However, we have
\begin{center}
\begin{tabular}{c|ccc ccc ccc}
    $D$
    &   -7 &   -8 &  -20
    &  -31 &  -35 &  -40
    &  -47 &  -56 &  -71
    \\\hline
    $L_D/L_{-4}$
    &  1.0 &   1.0 &  4.0
    &  1.0 &  16.0 &  4.0
    &  1.0 &   4.0 &  0.0
\end{tabular}
\end{center}
where $L_D := L(F_{249},1/2,\chi_D)\cdot\abs{D}$ and $\kro{D}{249}=+1$.

In this and the next section we will show where the squares in this table come from.
Before we do that, we show that for $F_{249}$ (as well as $F_{587}^-$
and $F_{713}^-$), we really do get $A(D)=0$ for all $D$.
\begin{lemma}\label{lem:AD}
Let $F$ be a paramodular form of weight $k$ and level $N$. 
Assume $F$ is an eigenform under the Atkin-Lehner operators $\mu_p$ for every $p\mid N$,
so that $F\!\mid\!\mu_p = \epsilon_p\,F$.
If $\epsilon_p=-1$ for any $p\mid N$ or if $k$ is odd, then $A_F(D)=0$ for all $D$.
\end{lemma}
\begin{proof}
For $N'\parallel N$, one can define an involution $W_{N'}$ over the set 
$\mathcal{Q}_N/\Gamma_0(N)$
(see \cite[p. 507]{GrossKohnenZagier}).
This involution is related to the Atkin-Lehner operators in the following way:
\[
a(W_{p^i}(T);F) = a(T;F\!\mid\!\mu_p) = \epsilon_p\,a(T;F)
\]
where $p^i$ is the largest power of $p$ dividing $N$.
Taking the sum over all classes $T\in\mathcal{Q}_{N,D}/\Gamma_0(N)$ shows that
$A(D) = \epsilon_p\,A(D)$, and it follows that $A(D)=0$ if
$\epsilon_p=-1$.
The case of odd $k$ is similar using $a(T[\smat{1&0\\0&-1}];F)=(-1)^k\,a(T;F)$.

\end{proof}


We will show now how to define a more refined average $B(D)$ on the
coefficients of $F$ for which  Lemma~\ref{lem:AD} does not apply.
In order to do that, 
we further decompose $\mathcal{Q}_{N,D}$ as follows.  Note that for
any $T=[Na,b,c]\in\mathcal{Q}_{N,D}$ we have $b^2\equiv D\pmod{4N}$
and we can thus define
\[
R_D := \{ \rho\mod 2N: \rho^2\equiv D\pmod{4N}\}.
\]
For each $\rho\in R_D$ we set
\[
\mathcal{Q}_{N,D,\rho} := \{ T=[Na,b,c] \in \mathcal{Q}_{N,D} : b\equiv
\rho\pmod{2N} \}.
\]
We observe that $\mathcal{Q}_{N,D}$ is the disjoint union of
$\mathcal{Q}_{N,D,\rho}$ for $\rho\in R_D$.
Now for each $\rho\in R_D$ we put
\[
B(D, \rho)=B_F(D,\rho):=\sum_{T\in\mathcal{Q}_{N,D,\rho}/\Gamma_0(N)}\frac{a(T;F)}{\varepsilon(T)}.
\]
\begin{lemma}\label{lem:ADandBD}  We note the following:
\begin{enumerate}
\item $A_F(D)=\frac12\,\sum_{\rho\in R_D} B_F(D,\rho)$,
\item $B_F(D,-\rho) = (-1)^k \, B_F(D,\rho)$, and
\item $\abs{B_F(D,\rho)}$ is independent of $\rho$.
\end{enumerate}
\end{lemma}

\begin{proof}
The first statement is obvious,
and the second follows from the fact that
$a(T[\smat{1&0\\0&-1}];F)=(-1)^k\,a(T;F)$.
The last statement follows by noting that the Atkin-Lehner
involutions $W_{N'}$ mentioned in the proof of Lemma~\ref{lem:AD}
transitively permute the sets $\mathcal{Q}_{N,D,\rho}$.
\end{proof}

Now we can finally define the new average:
\[
B(D) = B_F(D) := \frac12\,\sum_{\rho\in R_D} \abs{B(D,\rho)}.
\]
We observe by Lemma~\ref{lem:ADandBD} that
$B(D)=\frac12\,\abs{R_D}\,\abs{B(D,\rho)}$.
We also note that when $k$ is even and $N$ is prime,
we have $B(D)=\abs{A(D)}$ since $R_D=\{\pm\rho\}$ and
$B(D,-\rho)=B(D,\rho)$ in this case.

We return to the example $F_{249}$. We get the following table:
\begin{center}
\def\x{ }
\begin{tabular}{c|ccc ccc ccc}
    $D$
    &   -7 &   -8 &  -20
    &  -31 &  -35 &  -40
    &  -47 &  -56 &  -71
    \\\hline
    $L_D/L_{-4}$
    &  1.0 &   1.0 &  4.0
    &  1.0 &  16.0 &  4.0
    &  1.0 &   4.0 &  0.0
\\
   $B(D)/B(-8)$
    & \x &  1 &  2
    & \x &  4 & \x
    &  1 &  2 &  0
\end{tabular}
\end{center}
where again $L_D := L(F_{249},1/2,\chi_D)\cdot\abs{D}$ and $\kro{D}{249}=+1$.
When $\kro{D}{3}=\kro{D}{83}=-1$ the definition of $B(D)$ gives an empty sum;
this is indicated in the table above with an empty space. In the next
section we describe where the remaining squares come from.

\begin{remark}
Note that the value of $B(-71)=0$ is a non-trivial zero average,
predicting the vanishing of the twisted L-function at the center.
We will investigate such phenomena in a future paper.
\end{remark}




\subsection{A general case}\label{ssec:genus char}

In the previous section we looked at a form $F_{249}$ of composite
level for which the averages $A(D)$ vanish trivially.
We introduced a refined average $B(D)$ that explained some of the
data in the tables, but
in the case $\kro{D}{3}=\kro{D}{83}=-1$ the sum $B(D)$ is empty,
although the central values $L_D/L_{-4}$ are (nonzero) squares.

Consider the form $F_{587}^-$ as described in the Introduction.  It
was shown in \cite[Section 4]{RyanTornaria} that for the discriminants
$D$ so that $\kro{D}{587}=-1$ the sum $A(D)$ was empty and so, in
particular, our new sum $B(D)$ is also empty, and cannot explain
the fact that its normalized twisted central values are (nonzero) squares.

Also, in the definition of $A(D)$ and of $B(D)$ we require that $D$ be
negative, so neither average can make sense of the data in the
Introduction related to real quadratic twists of the L-functions of $F_{277}$.
In this section, using the genus theory for $\Gamma_0(N)$-classes of
quadratic forms, we fully explain these examples by
defining another new average $B_\ell(D)$ weighted by a genus character.  We
define this now. 

Fix a fundamental discriminant $\ell$. Then we define a genus character
$\chi_\ell$ similar to the generalized genus character
defined in~\cite{GrossKohnenZagier}.
Let $T=[Na,b,c]\in\mathcal{Q}_{N,\ell D}$ so that
$\gcd(a,b,c,\ell)=1$ and let $g=\gcd(N,b,c,\ell)$.  Define
$\tilde{T}=[Na/g,b,cg]$ and note that it represents an integer $n$ relatively
prime to $\ell$.  Now
\[
\chi_\ell(T) := \kro{\ell}{n} \prod_{p\mid g} s_p
\]
where 
\[
s_p=\begin{cases}
\kro{-\ell/p}{p} & p\text{ odd}\\
\kro{2}{t} & p=2,\, t\text{ the odd part of }\ell.
\end{cases}
\]
We note that $\chi_\ell$ has the following properties.  First, it is
completely multiplicative: if $\ell=\ell_1\ell_2$, then
$\chi_\ell=\chi_{\ell_1}\chi_{\ell_2}$.  Second, it behaves predictably
with respect to $W_p$.  Namely, if $p\nmid\ell$, then
$\chi_\ell(W_p\,T)=\kro{\ell}{p} \chi_\ell(T)$ and otherwise
$\chi_{p^\ast}(W_p\,T)=\chi_{p^\ast}(T)$ where
$p^\ast=\kro{-1}{p}\,p$ for odd $p$, and $p^\ast=-4$, $8$ or $-8$ for
$p=2$.

Then we define, for $D$ a fundamental discriminant
such that $\ell D<0$,
\[
B_\ell(D,\rho)= B_{\ell,F}(D,\rho)
:= \sum_{T\in\mathcal{Q}_{N,\ell D,\rho}/\Gamma_0(N)}\chi_{\ell}(T)\,\frac{a(T;F)}{\varepsilon(T)}
\]
and 
\[
B_\ell(D) =B_{\ell,F}(D) := \frac12\, \sum_{\rho\in R_{\ell D}}
B_\ell(D,\rho)\,.
\]
We note that $B_1(D)=B(D)$ as defined in the previous section.
One can also prove, using quadratic reciprocity, that
$B_\ell(D)=B_D(\ell)$.

\begin{remark}
If $\kro{\ell D}{p}=-1$ for some $p\mid N$, then $B_\ell(D)=0$,
because $\mathcal{Q}_{N,\ell D}$ is empty in this case.
Nevertheless, for any fundamental discriminant $D$, there exists
some $\ell$ for which $B_\ell(D)$ is not an empty sum.
\end{remark}

We will now complete the explanation of the examples we have discussed
so far. We start with the form $F_{249}$.
In the previous section we were able to explain the case
$\kro{D}{3}=\kro{D}{83}=+1$ with auxiliary discriminant $\ell=1$
(implicitly). We can explain the other case,
$\kro{D}{3}=\kro{D}{83}=-1$, by choosing $\ell=5$.
\begin{center}
\def\x{ }
\begin{tabular}{c|ccc ccc ccc}
    $D$
    &   -7 &   -8 &  -20
    &  -31 &  -35 &  -40
    &  -47 &  -56 &  -71
    \\\hline
    $L_D/L_{-4}$
    &  1.0 &   1.0 &  4.0
    &  1.0 &  16.0 &  4.0
    &  1.0 &   4.0 &  0.0
\\
   $B_1(D)/B_1(-8)$
    & \x &  1 &  2
    & \x &  4 & \x
    &  1 &  2 &  0
\\
   $B_5(D)/B_5(-4)$
    &  1 & \x & \x
    &  1 & \x &  2
    & \x & \x & \x
\end{tabular}
\end{center}
where $L_D := L(F_{249},1/2,\chi_D)\cdot\abs{D}$ and $\kro{D}{249}=+1$.
The empty entries correspond to empty sums as noted in the remark
above.

In order to explain the case of the form $F_{587}^-$, in the minus
space, we need to use an auxiliary discriminant $\ell>0$ such that
$\kro{\ell}{587}=-1$. Using $\ell=5$:
\begin{center}
\begin{tabular}{c|ccc ccc ccc ccc}
    $D$
    &  -4 &  -7 & -31
    & -40 & -43 & -47
    \\\hline
    $L_D/L_{-3}$
    &   1.0 &   1.0 &   4.0
    &   9.0 & 144.0 &   1.0
    \\
    $B_5(D)/B_5(-3)$
    &  1 &  1 &  2
    &  3 & 12 &  1
\end{tabular}
\end{center}
where $L_D := L(F_{587}^{-},1/2,\chi_D)\cdot\abs{D}$ and the table
shows fundamental discriminants for which $\kro{D}{587} = -1$.

Finally, in order to handle positive discriminants $D$, we can choose a
negative discriminant $\ell$. In the example of $F_{277}$ we choose
$\ell=-3$:
\begin{center}
\begin{tabular}{c|ccc ccc ccc ccc}
    $D$
    & 12 & 13 & 21
    & 28 & 29 & 40
    \\\hline
    $L_D/L_{1}$
    &  225.0 &  225.0 &  225.0
    &  225.0 & 2025.0 &  900.0
    \\
    $B_{-3}(D)/B_{-3}(1)$
    & 15 & 15 & 15
    & 15 & 45 & 30
\end{tabular}
\end{center}
where $L_D := L(F_{277},1/2,\chi_D)\cdot\abs{D}$ and $\kro{D}{277}=+1$.

\section{The Case of Nonlifts}\label{sec:nonlifts}

In the previous section, we highlighted some tables that give evidence
for our Conjectures~B and C.  Now we describe how those
tables were computed and how the tables in Section~\ref{sec:tables}
were computed.

The Paramodular Conjecture asserts that for each rational Hecke
eigenform $F$ that is not a Gritsenko lift, there
is an abelian surface $\mathcal{A}$ so that the Hasse-Weil L-function of
$\mathcal{A}$ and the spin L-function of $F$ are the same.
Suppose we have such an $F$ and such an $\mathcal{A}$. In all our
examples, $\mathcal{A}$ is isogenous to the Jacobian of a
hyperelliptic curve $C$. We list a sampling of such
hyperelliptic curves in Table~\ref{tbl:egs}, and more examples can be
found in \cite{RT}.

\begin{table}
\begin{center}
\begin{tabular}{l||c|r|r}
$F$ & $N$ & \hfill$C$\hfill{} & $T$
\\\hline\hline
$F_{249}$   & 249 & $y^2+(x^3 + 1)y=x^2 + x$ & 14 \\
$F_{277}$   & 277 & $y^2+y=x^5 - 2x^3 + 2x^2 - x$& 15\\
$F_{295}$   & 295 & $y^2+(x^3 + 1)y=-x^4 - x^3$&14\\
$F_{587}^-$ & 587 & $y^2+(x^3 + x + 1)y = -x^3 - x^2$&1\\
$F_{713}^+$ & 713 & $y^2+(x^3 + x + 1)y=-x^4$&9\\
$F_{713}^-$ & 713 & $y^2+(x^3 + x + 1)y=x^5 - x^3$&1\\
\end{tabular}
\end{center}
\caption{Hyperelliptic curves $C$ used to compute the L-series of the
  paramodular form $F$ associated to $C$ via the Paramodular
  Conjecture. Here $T$ denotes the torsion of the abelian surface $\text{Jac}(C)$.}
\label{tbl:egs}
\end{table}

Consider such a $C$.  Then the Euler product of $C$ can be found as in
\cite[Section 3]{RyanTornaria} and the functional equation can be
found, for example, in \cite{Dokchitser}, though we give it the
analytic normalization.  The central values were
then computed using Michael Rubinstein's \texttt{lcalc} \cite{lcalc}.

Now we describe how the averages are computed.  The Fourier
coefficients of the 6 paramodular forms whose L-functions correspond
to the L-functions of the curves listed in Table~\ref{tbl:egs} were
computed by Cris Poor and David Yuen.  The paramodular forms of level
$277$ and $587$ are publicly available and computed via the methods of
\cite{PoorYuen}.  The other four paramodular forms were computed by
Poor and Yuen for us, using an as of yet unpublished
method~\cite{PoorYuenPrivate}.

The sum $B_\ell(D)$ is computed using these Fourier coefficients using
a combination of Sage \cite{Sage} code and custom-written Python
code.  In particular, we implemented a class that represent binary quadratic
forms modulo $\Gamma_0(N)$ one of whose methods computes the
generalized genus character $\chi_{\ell}$. In
Table~\ref{tbl:verification} we summarize the forms and discriminants
for which we have computed both twisted averages and twisted central
values.

\begin{table}
\begin{center}
\begin{tabular}{l||c|r|l}
$F$ & $k_F$ & $\Delta_{\text{min}}$ & except these $\Delta$
\\\hline
$F_{249}$ & 0.831968 &  $-295$ & $\varnothing$ \\
$F_{277}$ & 0.537715 & $-2435$\rlap & $\{ -2167, -2180, -2191, -2200, -2212, -2215 \}$ \\
$F_{295}$ & 0.224744 &  $-276$ & $\{  -200,  -211,  -231,  -259 \}$ \\
$F_{587}^{-}$ & 0.002680 & $-1108$ & $\{  -927 \}$ \\
$F_{713}^{+}$ & 0.422121 &  $-260$ & $\varnothing$ \\
$F_{713}^{-}$ & 0.005248 &  $-260$ & $\varnothing$ \\
\end{tabular}
\end{center}
\captionsetup{singlelinecheck=off}
\caption[table description]{
Summary of forms and discriminants for which we have computed both
twisted averages $B_\ell(D)$ and the corresponding twisted central
values, and for which Conjecture~C has been numerically verified.
The discriminants for which we computed satisfy
$0 > \Delta \geq \Delta_{\text{min}}$,
where $\Delta=\ell D$, with the following exceptions:
\begin{itemize}[topsep=0pt,itemsep=0pt,parsep=0pt,label=--]
\item If a discriminant $\Delta$ is in the last column, it means that we did
not have all the Fourier coefficients necessary to compute the averages.
\item In the case of $F_{277}$, we have the further restriction
$\abs{\ell}\leq 500$ and $\abs{D}\leq 500$
due to loss of precision in computing $L_\ell$ and $L_D$.
\end{itemize}
}
\label{tbl:verification}
\end{table}

The following theorem summarizes the cases in which Conjecture~C has
been verified.
\begin{theorem}\label{thm:comp}
Let $F$ be one of the paramodular forms listed in
Table~\ref{tbl:verification}.
Let $\ell$ and $D$ be fundamental discriminants such that
$\ell D < 0$ satisfying the constraints described in the same table.
Then
\begin{equation*}
B_{\ell,F}(D)^2\approx
\alpha_{\ell D}\, k_F\,
L(F,1/2,\chi_\ell)\, L(F,1/2,\chi_D)\, \abs{D\ell}^{k-1}
\end{equation*}
numerically, with $k_F$ a positive constant listed in the table.
\end{theorem}

In addition to the cases listed in Table~\ref{tbl:verification}
we point out that more cases and more tables can be found at \cite{RT},
providing evidence for Conjecture~C using forms and curves that are not in this table.

\section{The Case of Lifts}\label{sec:lifts}

A Gritsenko lift $F$ \cite{Gritsenko} is a paramodular form that comes
from a Jacobi form $\phi$ which in turn corresponds to an elliptic modular
form $f$.
The standard reference for Jacobi forms is \cite{EichlerZagier} and we refer
the reader to \cite{PoorYuen} for background on the Gritsenko lift.

We will now state and prove a theorem that gives evidence for Conjecture~B in
the case of lifts.
\begin{theorem}\label{thm:grit}
Let $N$ be squarefree.  Suppose $F\in S^k(\para{N})$ is a Hecke eigenform
and a Gritsenko lift.
Let $D<0$ be a fundamental discriminant. Then
\begin{equation*}
B_{F}(D)^2=
\alpha_{D} \, C_{F} \, L(F,1/2,\chi_D) \, \abs{D}^{k-1}
\end{equation*}
where $C_F$ is a positive constant independent of $D$.
\end{theorem}

Let $F=\text{Grit}(\phi)$ where
\[\phi(\tau,z)=\sum_{n\geq 0}\sum_{r^2\leq 4nN} c(n,r)\,q^n\,\zeta^r\]
is a Jacobi form
of weight $k$ and index $N$.  We note \cite[Theorem 2.2,
p. 23]{EichlerZagier} that $c(n,r)$ depends only on $D=r^2-4nN$ and
$r\mod{2N}$; for each $\rho\in R_D$ we let
\[
c_\rho(D) := c\left(\frac{\rho^2-D}{4N}, \rho\right)
\qquad
\text{and}
\qquad
    c^*(D):=\frac{1}{2}\sum_{\rho\in R_D}\abs{c_\rho(D)} \,.
\]
We remark that $\abs{c_\rho(D)}$ is independent of $\rho$ and that
$c^*(D)$ 
is, up to sign, the
coefficient of a weight $k-1/2$ modular form \cite[Theorem~5.6, p. 69]{EichlerZagier}.


\begin{proposition} \label{prop:avg_lift}
If $D<0$ is a fundamental discriminant, then
\[
B_F(D)=c^{*}(D)\frac{h(D)}{w_D} \,.
\]
If $\ell\neq 1$ and $D\neq 1$ are fundamental discriminants, we have $B_{\ell,F}(D)=0$.
\end{proposition}
\begin{proof}
By the definition of the Gritsenko lift, we know that
$a(T ; F) = c_b(\disc T)$ for $T=[Na,b,c]\in\mathcal{Q}_{N}$,
provided $T$ is primitive, which is always the case for $\disc T$ fundamental.
Thus
\begin{align*}
\abs{B_\ell(D,\rho)}
& =\abs{\sum_{T\in\mathcal{Q}_{N,\ell D,\rho}/\Gamma_0(N)} \chi_\ell(T)\,\frac{a(T;F)}{\varepsilon(T)}}
\\
& = \abs{c_\rho(\ell D)}\,\abs{\sum_{T\in\mathcal{Q}_{N,\ell D,\rho}/\Gamma_0(N)} \chi_\ell(T)\,\frac{1}{\varepsilon(T)}}
\,.
\end{align*}
When $\ell=1$ the sum in the last term is
$\sum \frac{1}{\varepsilon(T)} = \frac{h(D)}{w_D}$,
since $\abs{\mathcal{Q}_{N,D,\rho}/\Gamma_0(N)}=h(D)$
for fundamental $D$ and $\varepsilon(T)=w_D$.
On the other hand if $l\neq 1$ and $D\neq 1$ then $\chi_\ell$ is a nontrivial character
in $\mathcal{Q}_{N,D,\rho}/\Gamma_0(N)$, hence the sum vanishes.
\end{proof}

Let $f$ be the elliptic modular form corresponding to the Jacobi form $\phi$ as
in \cite[Theorem 5]{SkoruppaZagier}.
It is a standard fact that $L(F,s) = \zeta(s+1/2)\, \zeta(s-1/2)\,
L(f,s)$ (using the analytic normalization, so that the center is at
$s=1/2$).
Twisting by $\chi_D$ we obtain
\[
   L(F,s,\chi_D) = L(s+1/2,\chi_D)\, L(s-1/2,\chi_D)\, L(f,s,\chi_D)
\]
valid on the region of convergence.
It follows from this that $L(F,s,\chi_D)$ has an analytic continuation
(with a pole at $s=3/2$ for $D=1$)
and, using Dirichlet's class number formula for the special values
$L(0,\chi_D)$ and $L(1,\chi_D)$, we have
\begin{proposition} \label{prop:centralvalue_lift}
\[
   L(F,1/2,\chi_D) =
\begin{cases}
\frac{4\pi^2}{w_D^2}\,\frac{h(D)^2}{\sqrt{\abs{D}}}\, L(f,1/2,\chi_D)
& \text{if $D<0$,} \\
0
& \text{if $D>1$,} \\
-\frac{1}{2} \, L'(f,1/2)
& \text{if $D=1$.} \\
\end{cases}
\]
\end{proposition}

\begin{proof}[Proof of Theorem~\ref{thm:grit}]
By Waldspurger's formula \cite{Waldspurger,Kohnen}, we have
\[
        c^*(D)^2 = \alpha_D\,k_f\,L(f,1/2,\chi_D)\,\abs{D}^{k-3/2}
\]
with $k_f>0$. The theorem thus follows from
Proposition~\ref{prop:avg_lift} and
Proposition~\ref{prop:centralvalue_lift},
with $k_F=k_f/4\pi^2$.
\end{proof}

\section{Torsion}\label{sec:torsion}


In the Introduction, we observed that for $L(F_{277},1/2,\chi_D)\cdot\abs{D}/L(F_{277},1/2)$ is
divisible by $15^2$ when $D>1$:
\begin{proposition}
Let $D>1$ and assume Conjecture~B.  Then, the ratio of special values
$L(F_{277},1/2,\chi_D)\cdot\abs{D}/L(F_{277},1/2)$ is divisible by $15^2$.
\end{proposition}
\begin{proof}
We recall \cite[Theorem 7.3]{PoorYuen} which asserts:  suppose $\phi$ is the first
Fourier-Jacobi coefficient of $F_{277}$ and let $G=\text{Grit}(\phi)$.
Then, for all $T\in \mathcal{Q}_{N}$,
\[
a(T;F)\equiv a(T;G)\pmod{15}.
\]
In Proposition~\ref{prop:avg_lift} we observed that $B_{l,G}(D) = 0$
when $\ell,D\neq 1$;
hence, it follows that $B_{l,F}(D)\equiv 0\pmod{15}$.
Finally, Conjecture~B implies that
\[
L(F,1/2,\chi_D) \cdot |D| / L(F,1/2) = \star\,  B_{-3}(D)^2 / B_{-3}(1)^2
\]
where $B_{-3}(1)=1$ (see Table~\ref{tbl:277}) and where $\star = 1$ if $p\mid D$ and 2 if $p\nmid D$.
\end{proof}



We recall (see Table~\ref{tbl:egs}) that the  Jacobian of the abelian surface
associated to $F_{277}$ by the Paramodular Conjecture has torsion of
size 15.  In \cite{PoorYuen}, it is suggested that this phenomenon
holds in generality.  Observe in Tables~\ref{tbl:249}--\ref{tbl:713m} each entry (both the
integers $B_\ell(D)$ and the normalized central values) is divisible
by the corresponding curve's torsion unless
$\ell=1$ or unless $D=1$.  This provides further (indirect) evidence for Poor and
Yuen's observation holding in general.



\bibliography{bocherer}
\bibliographystyle{alpha}

\section{Tables}\label{sec:tables}

\begin{table}
[h]
\caption{Data for the modular form $F_{249}$ based on the Hasse-Weil
    L-series for the hyperelliptic curve $y^{2} + \left(x^{3} +
      1\right) y = x^{2} + x$, whose Jacobian has 14-torsion.  The
    constant $C_\ell:= k_{249} \, L_\ell$ with $k_{249}=0.831968$.  The table
    displays the first few twists by real and by imaginary
    characters.  More comprehensive data for this curve can be found
    at \cite{RT}.  The values $L_\ell$ and $L_D$ are
    $L(F_{249},1/2,\chi_\ell)\cdot\abs{\ell}$ and $L(F_{249},1/2,\chi_D) \cdot\abs{D}$, respectively.
}
\label{tbl:249}
\begin{tabular}{>{\centering}p{15ex}|>{\hfill}p{8ex}>{\hfill}p{8ex}>{\hfill}p{8ex}>{\hfill}p{8ex}>{\hfill}p{8ex}>{\hfill}p{8ex}}
                           $D$  &           -8  &          -20  &          -35  &          -47  &          -56  &          -71 \\\hline
  $\alpha_{D} \, C_{1} \, L_D$  &          4.0  &         16.0  &         64.0  &          4.0  &         16.0  &          0.0 \\
                    $B_{1}(D)$  &            2  &            4  &            8  &            2  &            4  &            0 \\
\end{tabular}
\bigskip\par
\begin{tabular}{>{\centering}p{15ex}|>{\hfill}p{8ex}>{\hfill}p{8ex}>{\hfill}p{8ex}>{\hfill}p{8ex}>{\hfill}p{8ex}>{\hfill}p{8ex}}
                           $D$  &           -3  &           -4  &           -7  &          -31  &          -40  &          -51 \\\hline
 $\alpha_{5D} \, C_{5} \, L_D$  &        196.0  &        784.0  &        784.0  &        784.0  &       3136.0  &      19600.0 \\
                    $B_{5}(D)$  &           14  &           28  &           28  &           28  &           56  &          140 \\
\end{tabular}
\bigskip\par
\begin{tabular}{>{\centering}p{15ex}|>{\hfill}p{8ex}>{\hfill}p{8ex}>{\hfill}p{8ex}>{\hfill}p{8ex}>{\hfill}p{8ex}>{\hfill}p{8ex}}
                           $D$  &           -3  &           -4  &           -7  &          -31  &          -40  &          -51 \\\hline
 $\alpha_{8D} \, C_{8} \, L_D$  &        196.0  &        784.0  &        784.0  &        784.0  &       3136.0  &      19600.0 \\
                    $B_{8}(D)$  &           14  &           28  &           28  &           28  &           --  &           -- \\
\end{tabular}
\bigskip\par
\begin{tabular}{>{\centering}p{15ex}|>{\hfill}p{8ex}>{\hfill}p{8ex}>{\hfill}p{8ex}>{\hfill}p{8ex}>{\hfill}p{8ex}>{\hfill}p{8ex}}
                           $D$  &            5  &            8  &           24  &           53  &           56  &           60 \\\hline
$\alpha_{-3D} \, C_{-3} \, L_D$  &        196.0  &        196.0  &        784.0  &       3136.0  &       3136.0  &       3136.0 \\
                   $B_{-3}(D)$  &           14  &           14  &           28  &           56  &           56  &           56 \\
\end{tabular}
\bigskip\par
\begin{tabular}{>{\centering}p{15ex}|>{\hfill}p{8ex}>{\hfill}p{8ex}>{\hfill}p{8ex}>{\hfill}p{8ex}>{\hfill}p{8ex}>{\hfill}p{8ex}}
                           $D$  &            5  &            8  &           24  &           53  &           56  &           57 \\\hline
$\alpha_{-4D} \, C_{-4} \, L_D$  &        784.0  &        784.0  &        784.0  &      12544.0  &      12544.0  &          0.0 \\
                   $B_{-4}(D)$  &           28  &           28  &           28  &          112  &          112  &            0 \\
\end{tabular}
\bigskip\par
\begin{tabular}{>{\centering}p{15ex}|>{\hfill}p{8ex}>{\hfill}p{8ex}>{\hfill}p{8ex}>{\hfill}p{8ex}>{\hfill}p{8ex}>{\hfill}p{8ex}}
                           $D$  &            5  &            8  &           24  &           53  &           56  &           57 \\\hline
$\alpha_{-7D} \, C_{-7} \, L_D$  &        784.0  &        784.0  &        784.0  &      12544.0  &      12544.0  &          0.0 \\
                   $B_{-7}(D)$  &           28  &           28  &           28  &           --  &           --  &           -- \\
\end{tabular}
\bigskip\par
\begin{tabular}{>{\centering}p{15ex}|>{\hfill}p{8ex}>{\hfill}p{8ex}>{\hfill}p{8ex}>{\hfill}p{8ex}>{\hfill}p{8ex}>{\hfill}p{8ex}}
                           $D$  &            1  &           28  &           37  &           40  &           61  &          109 \\\hline
$\alpha_{-8D} \, C_{-8} \, L_D$  &          4.0  &          0.0  &       3136.0  &       3136.0  &       3136.0  &      28224.0 \\
                   $B_{-8}(D)$  &            2  &            0  &           --  &           --  &           --  &           -- \\
\end{tabular}
\bigskip\par
\end{table}

\begin{table}
\caption{Data for the modular form $F_{277}$ based on the Hasse-Weil
    L-series for the hyperelliptic curve $y^{2} + y = x^{5} - 2 \,
    x^{3} + 2 \, x^{2} - x$, whose Jacobian has 15-torsion.  The
    constant $C_\ell:= k_{277} \, L_\ell$ with $k_{277}=0.537716$.  The table
    displays the first few twists by real and by imaginary
    characters.  More comprehensive data for this curve can be found
    at \cite{RT}.  The values $L_\ell$ and $L_D$ are
    $L(F_{277},1/2,\chi_\ell) \cdot\abs{\ell}$ and $L(F_{277},1/2,\chi_D) \cdot\abs{D}$, respectively. 
}
\label{tbl:277}
\begin{tabular}{>{\centering}p{15ex}|>{\hfill}p{8ex}>{\hfill}p{8ex}>{\hfill}p{8ex}>{\hfill}p{8ex}>{\hfill}p{8ex}>{\hfill}p{8ex}}
                           $D$  &           -3  &           -4  &           -7  &          -19  &          -23  &          -39 \\\hline
  $\alpha_{D} \, C_{1} \, L_D$  &          1.0  &          1.0  &          1.0  &          4.0  &         -0.0  &          1.0 \\
                    $B_{1}(D)$  &            1  &            1  &            1  &            2  &            0  &            1 \\
\end{tabular}
\bigskip\par
\begin{tabular}{>{\centering}p{15ex}|>{\hfill}p{8ex}>{\hfill}p{8ex}>{\hfill}p{8ex}>{\hfill}p{8ex}>{\hfill}p{8ex}>{\hfill}p{8ex}}
                           $D$  &           -3  &           -4  &           -7  &          -19  &          -23  &          -39 \\\hline
$\alpha_{12D} \, C_{12} \, L_D$  &        225.0  &        225.0  &        225.0  &        900.0  &         -0.0  &        225.0 \\
                   $B_{12}(D)$  &           15  &           15  &           15  &           30  &            0  &           15 \\
\end{tabular}
\bigskip\par
\begin{tabular}{>{\centering}p{15ex}|>{\hfill}p{8ex}>{\hfill}p{8ex}>{\hfill}p{8ex}>{\hfill}p{8ex}>{\hfill}p{8ex}>{\hfill}p{8ex}}
                           $D$  &           -3  &           -4  &           -7  &          -19  &          -23  &          -39 \\\hline
$\alpha_{13D} \, C_{13} \, L_D$  &        225.0  &        225.0  &        225.0  &        900.0  &         -0.0  &        225.0 \\
                   $B_{13}(D)$  &           15  &           15  &           15  &           30  &            0  &           15 \\
\end{tabular}
\bigskip\par
\begin{tabular}{>{\centering}p{15ex}|>{\hfill}p{8ex}>{\hfill}p{8ex}>{\hfill}p{8ex}>{\hfill}p{8ex}>{\hfill}p{8ex}>{\hfill}p{8ex}}
                           $D$  &            1  &           12  &           13  &           21  &           28  &           29 \\\hline
$\alpha_{-3D} \, C_{-3} \, L_D$  &          1.0  &        225.0  &        225.0  &        225.0  &        225.0  &       2025.0 \\
                   $B_{-3}(D)$  &            1  &           15  &           15  &           15  &           15  &           45 \\
\end{tabular}
\bigskip\par
\begin{tabular}{>{\centering}p{15ex}|>{\hfill}p{8ex}>{\hfill}p{8ex}>{\hfill}p{8ex}>{\hfill}p{8ex}>{\hfill}p{8ex}>{\hfill}p{8ex}}
                           $D$  &            1  &           12  &           13  &           21  &           28  &           29 \\\hline
$\alpha_{-4D} \, C_{-4} \, L_D$  &          1.0  &        225.0  &        225.0  &        225.0  &        225.0  &       2025.0 \\
                   $B_{-4}(D)$  &            1  &           15  &           15  &           15  &           15  &           45 \\
\end{tabular}
\bigskip\par
\begin{tabular}{>{\centering}p{15ex}|>{\hfill}p{8ex}>{\hfill}p{8ex}>{\hfill}p{8ex}>{\hfill}p{8ex}>{\hfill}p{8ex}>{\hfill}p{8ex}}
                           $D$  &            1  &           12  &           13  &           21  &           28  &           29 \\\hline
$\alpha_{-7D} \, C_{-7} \, L_D$  &          1.0  &        225.0  &        225.0  &        225.0  &        225.0  &       2025.0 \\
                   $B_{-7}(D)$  &            1  &           15  &           15  &           15  &           15  &           45 \\
\end{tabular}
\bigskip\par
\end{table}

\begin{table}
\caption{Data for the modular form $F_{295}$ based on the Hasse-Weil
    L-series for the hyperelliptic curve $y^{2} + \left(x^{3} +
      1\right) y = -x^{4} - x^{3}$, whose Jacobian has 14-torsion.  The
    constant $C_\ell:= k_{295} \, L_\ell$ with $k_{295}=0.224745$.  The table
    displays the first few twists by real and by imaginary
    characters.  More comprehensive data for this curve can be found
    at \cite{RT}.  The values $L_\ell$ and $L_D$ are
    $L(F_{295},1/2,\chi_\ell) \cdot\abs{\ell}$ and $L(F_{295},1/2,\chi_D) \cdot\abs{D}$, respectively.
}
\label{tbl:295}
\begin{tabular}{>{\centering}p{15ex}|>{\hfill}p{8ex}>{\hfill}p{8ex}>{\hfill}p{8ex}>{\hfill}p{8ex}>{\hfill}p{8ex}>{\hfill}p{8ex}}
                           $D$  &          -11  &          -24  &          -31  &          -39  &          -40  &          -55 \\\hline
  $\alpha_{D} \, C_{1} \, L_D$  &          4.0  &          4.0  &          4.0  &          4.0  &         16.0  &          4.0 \\
                    $B_{1}(D)$  &            2  &            2  &            2  &            2  &            4  &            2 \\
\end{tabular}
\bigskip\par
\begin{tabular}{>{\centering}p{15ex}|>{\hfill}p{8ex}>{\hfill}p{8ex}>{\hfill}p{8ex}>{\hfill}p{8ex}>{\hfill}p{8ex}>{\hfill}p{8ex}}
                           $D$  &          -11  &          -24  &          -31  &          -39  &          -55  &          -56 \\\hline
 $\alpha_{5D} \, C_{5} \, L_D$  &        196.0  &        196.0  &        196.0  &        196.0  &        784.0  &        196.0 \\
                    $B_{5}(D)$  &           14  &           14  &           14  &           14  &           28  &           -- \\
\end{tabular}
\bigskip\par
\begin{tabular}{>{\centering}p{15ex}|>{\hfill}p{8ex}>{\hfill}p{8ex}>{\hfill}p{8ex}>{\hfill}p{8ex}>{\hfill}p{8ex}>{\hfill}p{8ex}}
                           $D$  &           -3  &           -7  &          -68  &          -87  &          -88  &         -107 \\\hline
 $\alpha_{8D} \, C_{8} \, L_D$  &        196.0  &        196.0  &       3136.0  &        784.0  &       3136.0  &      15876.0 \\
                    $B_{8}(D)$  &           14  &           14  &           --  &           --  &           --  &           -- \\
\end{tabular}
\bigskip\par
\begin{tabular}{>{\centering}p{15ex}|>{\hfill}p{8ex}>{\hfill}p{8ex}>{\hfill}p{8ex}>{\hfill}p{8ex}>{\hfill}p{8ex}>{\hfill}p{8ex}}
                           $D$  &            8  &           13  &           33  &           37  &           73  &           77 \\\hline
$\alpha_{-3D} \, C_{-3} \, L_D$  &        196.0  &        196.0  &         -0.0  &       1764.0  &         -0.0  &        784.0 \\
                   $B_{-3}(D)$  &           14  &           14  &            0  &           42  &            0  &           -- \\
\end{tabular}
\bigskip\par
\begin{tabular}{>{\centering}p{15ex}|>{\hfill}p{8ex}>{\hfill}p{8ex}>{\hfill}p{8ex}>{\hfill}p{8ex}>{\hfill}p{8ex}>{\hfill}p{8ex}}
                           $D$  &            8  &           13  &           33  &           37  &           73  &           77 \\\hline
$\alpha_{-7D} \, C_{-7} \, L_D$  &        196.0  &        196.0  &         -0.0  &       1764.0  &         -0.0  &        784.0 \\
                   $B_{-7}(D)$  &           14  &           14  &           --  &           --  &           --  &           -- \\
\end{tabular}
\bigskip\par
\begin{tabular}{>{\centering}p{15ex}|>{\hfill}p{8ex}>{\hfill}p{8ex}>{\hfill}p{8ex}>{\hfill}p{8ex}>{\hfill}p{8ex}>{\hfill}p{8ex}}
                           $D$  &            1  &            5  &           21  &           29  &           41  &           60 \\\hline
$\alpha_{-11D} \, C_{-11} \, L_D$  &          4.0  &        196.0  &        784.0  &       3136.0  &        784.0  &       3136.0 \\
                  $B_{-11}(D)$  &            2  &           14  &           --  &           --  &           --  &           -- \\
\end{tabular}
\bigskip\par
\end{table}

\begin{table}
\caption{Data for the modular form $F_{587}^{-}$ based on the Hasse-Weil
    L-series for the hyperelliptic curve $y^{2} + \left(x^{3} + x + 1\right) y = -x^{3} - x^{2}$, whose Jacobian has 1-torsion.  The
    constant $C_\ell:= k_{587}^{-} \, L_\ell$ with
    $k_{587}^{-}=0.002681$.  
The table
    displays the first few twists by real and by imaginary
    characters.  More comprehensive data for this curve can be found
    at \cite{RT}.  The values $L_\ell$ and $L_D$ are
    $L(F_{587}^-,1/2,\chi_\ell) \cdot\abs{\ell}$ and $L(F_{587}^-,1/2,\chi_D) \cdot\abs{D}$, respectively.
}
\label{tbl:587m}
\begin{tabular}{>{\centering}p{15ex}|>{\hfill}p{8ex}>{\hfill}p{8ex}>{\hfill}p{8ex}>{\hfill}p{8ex}>{\hfill}p{8ex}>{\hfill}p{8ex}}
                           $D$  &           -3  &           -4  &           -7  &          -31  &          -40  &          -43 \\\hline
 $\alpha_{5D} \, C_{5} \, L_D$  &          4.0  &          4.0  &          4.0  &         16.0  &         36.0  &        576.0 \\
                    $B_{5}(D)$  &            2  &            2  &            2  &            4  &            6  &           24 \\
\end{tabular}
\bigskip\par
\begin{tabular}{>{\centering}p{15ex}|>{\hfill}p{8ex}>{\hfill}p{8ex}>{\hfill}p{8ex}>{\hfill}p{8ex}>{\hfill}p{8ex}>{\hfill}p{8ex}}
                           $D$  &           -3  &           -4  &           -7  &          -31  &          -40  &          -43 \\\hline
 $\alpha_{8D} \, C_{8} \, L_D$  &          4.0  &          4.0  &          4.0  &         16.0  &         36.0  &        576.0 \\
                    $B_{8}(D)$  &            2  &            2  &            2  &            4  &            6  &           24 \\
\end{tabular}
\bigskip\par
\begin{tabular}{>{\centering}p{15ex}|>{\hfill}p{8ex}>{\hfill}p{8ex}>{\hfill}p{8ex}>{\hfill}p{8ex}>{\hfill}p{8ex}>{\hfill}p{8ex}}
                           $D$  &           -3  &           -4  &           -7  &          -31  &          -40  &          -43 \\\hline
$\alpha_{13D} \, C_{13} \, L_D$  &          4.0  &          4.0  &          4.0  &         16.0  &         36.0  &        576.0 \\
                   $B_{13}(D)$  &            2  &            2  &            2  &            4  &            6  &           24 \\
\end{tabular}
\bigskip\par
\begin{tabular}{>{\centering}p{15ex}|>{\hfill}p{8ex}>{\hfill}p{8ex}>{\hfill}p{8ex}>{\hfill}p{8ex}>{\hfill}p{8ex}>{\hfill}p{8ex}}
                           $D$  &            5  &            8  &           13  &           24  &           33  &           37 \\\hline
$\alpha_{-3D} \, C_{-3} \, L_D$  &          4.0  &          4.0  &          4.0  &          4.0  &          4.0  &         16.0 \\
                   $B_{-3}(D)$  &            2  &            2  &            2  &            2  &            2  &            4 \\
\end{tabular}
\bigskip\par
\begin{tabular}{>{\centering}p{15ex}|>{\hfill}p{8ex}>{\hfill}p{8ex}>{\hfill}p{8ex}>{\hfill}p{8ex}>{\hfill}p{8ex}>{\hfill}p{8ex}}
                           $D$  &            5  &            8  &           13  &           24  &           33  &           37 \\\hline
$\alpha_{-4D} \, C_{-4} \, L_D$  &          4.0  &          4.0  &          4.0  &          4.0  &          4.0  &         16.0 \\
                   $B_{-4}(D)$  &            2  &            2  &            2  &            2  &            2  &            4 \\
\end{tabular}
\bigskip\par
\begin{tabular}{>{\centering}p{15ex}|>{\hfill}p{8ex}>{\hfill}p{8ex}>{\hfill}p{8ex}>{\hfill}p{8ex}>{\hfill}p{8ex}>{\hfill}p{8ex}}
                           $D$  &            5  &            8  &           13  &           24  &           33  &           37 \\\hline
$\alpha_{-7D} \, C_{-7} \, L_D$  &          4.0  &          4.0  &          4.0  &          4.0  &          4.0  &         16.0 \\
                   $B_{-7}(D)$  &            2  &            2  &            2  &            2  &            2  &            4 \\
\end{tabular}
\bigskip\par
\end{table}

\begin{table}
\caption{Data for the modular form $F_{713}^{+}$ based on the Hasse-Weil
    L-series for the hyperelliptic curve $y^{2} + \left(x^{3} + x + 1\right) y = -x^{4}$, whose Jacobian has 9-torsion.  The
    constant $C_\ell:= k_{713}^{+} \, L_\ell$ with
    $k_{713}^{+}=0.422122$.  
The table
    displays the first few twists by real and by imaginary
    characters.  More comprehensive data for this curve can be found
    at \cite{RT}.  The values $L_\ell$ and $L_D$ are
    $L(F_{713}^+,1/2,\chi_\ell) \cdot\abs{\ell}$ and $L(F_{713}^+,1/2,\chi_D) \cdot\abs{D}$, respectively. 
}
\label{tbl:713p}
\begin{tabular}{>{\centering}p{15ex}|>{\hfill}p{8ex}>{\hfill}p{8ex}>{\hfill}p{8ex}>{\hfill}p{8ex}>{\hfill}p{8ex}>{\hfill}p{8ex}}
                           $D$  &          -11  &          -15  &          -23  &          -43  &          -68  &          -79 \\\hline
  $\alpha_{D} \, C_{1} \, L_D$  &         16.0  &         16.0  &         -0.0  &        144.0  &         64.0  &         64.0 \\
                    $B_{1}(D)$  &            4  &            4  &            0  &           12  &            8  &            8 \\
\end{tabular}
\bigskip\par
\begin{tabular}{>{\centering}p{15ex}|>{\hfill}p{8ex}>{\hfill}p{8ex}>{\hfill}p{8ex}>{\hfill}p{8ex}>{\hfill}p{8ex}>{\hfill}p{8ex}}
                           $D$  &          -11  &          -15  &          -23  &          -43  &          -68  &          -79 \\\hline
 $\alpha_{8D} \, C_{8} \, L_D$  &       1296.0  &       1296.0  &         -0.0  &      11664.0  &       5184.0  &       5184.0 \\
                    $B_{8}(D)$  &           36  &           36  &            0  &           --  &           --  &           -- \\
\end{tabular}
\bigskip\par
\begin{tabular}{>{\centering}p{15ex}|>{\hfill}p{8ex}>{\hfill}p{8ex}>{\hfill}p{8ex}>{\hfill}p{8ex}>{\hfill}p{8ex}>{\hfill}p{8ex}}
                           $D$  &           -4  &           -8  &          -35  &          -39  &          -47  &          -59 \\\hline
$\alpha_{17D} \, C_{17} \, L_D$  &          0.0  &          0.0  &          0.0  &         -0.0  &          0.0  &         -0.0 \\
                   $B_{17}(D)$  &            0  &            0  &           --  &           --  &           --  &           -- \\
\end{tabular}
\bigskip\par
\begin{tabular}{>{\centering}p{15ex}|>{\hfill}p{8ex}>{\hfill}p{8ex}>{\hfill}p{8ex}>{\hfill}p{8ex}>{\hfill}p{8ex}>{\hfill}p{8ex}}
                           $D$  &           17  &           21  &           37  &           44  &           53  &           57 \\\hline
$\alpha_{-4D} \, C_{-4} \, L_D$  &          0.0  &       1296.0  &         -0.0  &       1296.0  &       1296.0  &         -0.0 \\
                   $B_{-4}(D)$  &            0  &           36  &            0  &           36  &           36  &            0 \\
\end{tabular}
\bigskip\par
\begin{tabular}{>{\centering}p{15ex}|>{\hfill}p{8ex}>{\hfill}p{8ex}>{\hfill}p{8ex}>{\hfill}p{8ex}>{\hfill}p{8ex}>{\hfill}p{8ex}}
                           $D$  &           17  &           21  &           37  &           44  &           53  &           57 \\\hline
$\alpha_{-8D} \, C_{-8} \, L_D$  &          0.0  &       1296.0  &         -0.0  &       1296.0  &       1296.0  &         -0.0 \\
                   $B_{-8}(D)$  &            0  &           36  &           --  &           --  &           --  &           -- \\
\end{tabular}
\bigskip\par
\begin{tabular}{>{\centering}p{15ex}|>{\hfill}p{8ex}>{\hfill}p{8ex}>{\hfill}p{8ex}>{\hfill}p{8ex}>{\hfill}p{8ex}>{\hfill}p{8ex}}
                           $D$  &            1  &            8  &           41  &           69  &           93  &          101 \\\hline
$\alpha_{-11D} \, C_{-11} \, L_D$  &         16.0  &       1296.0  &       1296.0  &      20736.0  &      20736.0  &      20736.0 \\
                  $B_{-11}(D)$  &            4  &           36  &           --  &           --  &           --  &           -- \\
\end{tabular}
\bigskip\par
\end{table}

\begin{table}
\caption{Data for the modular form $F_{713}^{-}$ based on the Hasse-Weil
    L-series for the hyperelliptic curve $y^{2} + \left(x^{3} + x + 1\right) y = x^{5} - x^{3}$, whose Jacobian has 1-torsion.  The
    constant $C_\ell:= k_{713}^{-} \, L_\ell$ with
    $k_{713}^{-}=0.005249$.  
The table
    displays the first few twists by real and by imaginary
    characters.  More comprehensive data for this curve can be found
    at \cite{RT}.  The values $L_\ell$ and $L_D$ are
    $L(F_{713}^-,1/2,\chi_\ell) \cdot\abs{\ell}$ and $L(F_{713}^-,1/2,\chi_D) \cdot\abs{D}$, respectively.
}
\label{tbl:713m}
\begin{tabular}{>{\centering}p{15ex}|>{\hfill}p{8ex}>{\hfill}p{8ex}>{\hfill}p{8ex}>{\hfill}p{8ex}>{\hfill}p{8ex}>{\hfill}p{8ex}}
                           $D$  &           -3  &          -24  &          -52  &          -55  &         -104  &         -116 \\\hline
 $\alpha_{5D} \, C_{5} \, L_D$  &         16.0  &         16.0  &        400.0  &         64.0  &         16.0  &        144.0 \\
                    $B_{5}(D)$  &            4  &            4  &           20  &           --  &           --  &           -- \\
\end{tabular}
\bigskip\par
\begin{tabular}{>{\centering}p{15ex}|>{\hfill}p{8ex}>{\hfill}p{8ex}>{\hfill}p{8ex}>{\hfill}p{8ex}>{\hfill}p{8ex}>{\hfill}p{8ex}}
                           $D$  &           -7  &          -19  &          -20  &          -40  &          -51  &          -56 \\\hline
$\alpha_{12D} \, C_{12} \, L_D$  &         16.0  &        256.0  &         16.0  &        400.0  &        576.0  &         16.0 \\
                   $B_{12}(D)$  &            4  &           16  &            4  &           --  &           --  &           -- \\
\end{tabular}
\bigskip\par
\begin{tabular}{>{\centering}p{15ex}|>{\hfill}p{8ex}>{\hfill}p{8ex}>{\hfill}p{8ex}>{\hfill}p{8ex}>{\hfill}p{8ex}>{\hfill}p{8ex}}
                           $D$  &           -7  &          -19  &          -20  &          -40  &          -51  &          -56 \\\hline
$\alpha_{13D} \, C_{13} \, L_D$  &         16.0  &        256.0  &         16.0  &        400.0  &        576.0  &         16.0 \\
                   $B_{13}(D)$  &            4  &           16  &            4  &           --  &           --  &           -- \\
\end{tabular}
\bigskip\par
\begin{tabular}{>{\centering}p{15ex}|>{\hfill}p{8ex}>{\hfill}p{8ex}>{\hfill}p{8ex}>{\hfill}p{8ex}>{\hfill}p{8ex}>{\hfill}p{8ex}}
                           $D$  &            5  &           28  &           33  &           40  &           56  &           76 \\\hline
$\alpha_{-3D} \, C_{-3} \, L_D$  &         16.0  &         16.0  &          0.0  &         16.0  &         16.0  &         -0.0 \\
                   $B_{-3}(D)$  &            4  &            4  &            0  &            4  &            4  &            0 \\
\end{tabular}
\bigskip\par
\begin{tabular}{>{\centering}p{15ex}|>{\hfill}p{8ex}>{\hfill}p{8ex}>{\hfill}p{8ex}>{\hfill}p{8ex}>{\hfill}p{8ex}>{\hfill}p{8ex}}
                           $D$  &           12  &           13  &           24  &           29  &           73  &           77 \\\hline
$\alpha_{-7D} \, C_{-7} \, L_D$  &         16.0  &         16.0  &         16.0  &        144.0  &          0.0  &        576.0 \\
                   $B_{-7}(D)$  &            4  &            4  &            4  &           12  &           --  &           -- \\
\end{tabular}
\bigskip\par
\begin{tabular}{>{\centering}p{15ex}|>{\hfill}p{8ex}>{\hfill}p{8ex}>{\hfill}p{8ex}>{\hfill}p{8ex}>{\hfill}p{8ex}>{\hfill}p{8ex}}
                           $D$  &           12  &           13  &           24  &           29  &           73  &           77 \\\hline
$\alpha_{-19D} \, C_{-19} \, L_D$  &        256.0  &        256.0  &        256.0  &       2304.0  &          0.0  &       9216.0 \\
                  $B_{-19}(D)$  &           16  &           16  &           --  &           --  &           --  &           -- \\
\end{tabular}
\bigskip\par
\end{table}

\end{document}